\documentclass[a4paper]{article}

\pdfoutput=1

\usepackage[english]{babel}
\usepackage[utf8]{inputenc}
\usepackage[T1]{fontenc}

\usepackage[a4paper,top=3cm,bottom=2cm,left=3cm,right=3cm,marginparwidth=2cm]{geometry}

\usepackage{amsfonts}
\usepackage{amsmath}
\usepackage{amsthm}
\usepackage{amssymb}
\usepackage{graphicx}
\usepackage[colorinlistoftodos]{todonotes}
\usepackage[colorlinks=true, allcolors=blue]{hyperref}

\usepackage[
    style=ext-alphabetic,
    backend=biber, 
    sorting=nyt, 
    articlein=false]{biblatex}
\renewbibmacro{in:}{}
\usepackage{csquotes}

\begin{document}
\theoremstyle{plain}
\newtheorem{theorem}{Theorem}[section]
\newtheorem{lemma}[theorem]{Lemma}
\newtheorem{corollary}[theorem]{Corollary}
\newtheorem{proposition}[theorem]{Proposition}
\newtheorem{conjecture}[theorem]{Conjecture}
\newtheorem{criterion}[theorem]{Criterion}
\newtheorem{algorithm}[theorem]{Algorithm}
\newtheorem{condition}[theorem]{Condition}
\newtheorem{problem}[theorem]{Problem}
\newtheorem{example}[theorem]{Example}
\newtheorem{exercise}{Exercise}[section]
\newtheorem{obs}{Observation}
\newtheorem{note}[theorem]{Note}
\newtheorem{notation}[theorem]{Notation}
\newtheorem{claim}[theorem]{Claim}
\newtheorem{summary}[theorem]{Summary}
\newtheorem{acknowledgment}[theorem]{Acknowledgment}
\newtheorem{case[theorem]}{Case}
\newtheorem{conclusion}[theorem]{Conclusion}

\theoremstyle{definition}
\newtheorem{definition}{Definition}
\theoremstyle{remark}
\newtheorem{remark}{Remark}

\def\bt{\mathbf t}
\def\bA{\mathbb{A}}
\def\bB{\mathbb{B}}
\def\bN{\mathbb{N}}

\author{Aliaksei Semchankau}
\title{On Differences of Multiplicative Functions and Solutions of the Equation $n - \varphi(n)=c$
\footnote{This work is supported by the Russian Science Foundation under grant 19-11-00001.}
}

\maketitle

\begin{abstract}

We will study the solutions to the equation $f(n) - g(n) = c$, where $f$ and $g$ are multiplicative functions and $c$ is a constant. More precisely, we prove that the number of solutions does not exceed $c^{1-\epsilon}$ when $f, g$ and solutions $n$ satisfy some certain constraints, such as $f(n) > g(n)$ for $n > 1$. In particular, we will prove the following estimate: the number of solutions to the equation $n - \varphi(n) = c$ is:
$$
G(c + 1) + O(c^{0.75 + o(1)}),
$$
where $G(k)$ is the number of ways to represent $k$ as a sum of two primes.
This result is based on some properties of configurations of points and lines.

\end{abstract}

\section{Introduction}

Let $\varphi$ be the Euler function, which value at $n$ is defined as

$$
\varphi(n) = 
n\prod_{p | n} 
\left(
1 - \frac{1}{p}
\right).
$$
For a given $c$ the equation $\varphi(n) = c$ has been studied in works of Erd\H{o}s \cite{Erd35} and Pomerance \cite{Pom80}. In particular, were obtained following bounds for $T(c) = |\{ n: \varphi(n) = c\}|$:

$$
T(c) \leqslant ce^
{-\left(1 + o(1)\right)
\frac{\log{c}\log\log\log{c}}
{\log\log{c}}},
$$
and also
$$
T(c) \geqslant c^{\alpha}, 
$$
 for infinitely many $c$ and $\alpha = 0.55655\ldots$

 One can consider the cototient \cite{BL04} Euler function:
 
$$
\psi (n) = n - \varphi(n).
$$

For a given $c$ the equation $\psi(n) = c$ has been studied in the work of Banks and Luca \cite{BL04}. In particular it was demonstrated that for almost all (i.e. of density $1$) primes $p$ equation $\psi(n) = 2p$ can not be solved in $n$. We show the following. Let $G(n)$ be the amount of ways to represent $n$ as a sum of $2$ prime numbers. Then we have:
 
\begin{theorem}
 The amount of solutions $n$ to $\psi(n) = c$ for given $c > 1$ equals to
 $$
 G(c + 1) + 
 O(
 c^{0.75 + o(1)}
 ).
 $$ 
\end{theorem}
It is easy to see, that if $c$ is an even number, then the term $G(c+1)$ in the formula above is $O(1)$.

\section{On Special Configurations of Points and Lines}

In this paragraph we consider special configurations of points and lines in Euclidean plane, which occur to be useful while studying differences $f(n) - g(n)$, where both $f$, $g$ are multiplicative functions. 

\medskip

It was demonstrated in \cite{ST83}:
\begin{theorem} [\bf Szemeredi-Trotter]
Given $n$ points and $m$ lines in the Euclidean plane, the number of incidences (i.e. the number of pairs point-line such that point belongs to line) is $O\big((mn)^{\frac{2}{3}} + m + n\big)$.
\end{theorem}

Clearly, one can associate a configuration of points and lines with a bipartite graph, where one part is points, and another is lines. Point and line are connected with an edge if and only if point belongs to line. 
The theorem above states that when part sizes are equal $m$ and $n$, the number of edges does not exceed $O\big((mn)^{\frac{2}{3}} + m + n\big)$.

It turns out that under some particular restrictions this bound might be improved to `almost-linear' one on $m + n$. To carry this out, we would need some definitions.

\begin{definition}
We call configuration of points $P$ and lines $L$ \emph{normed natural}, if there exist such natural numbers $A_1, a_1, A_2, a_2, \cdots, A_{|P|}, a_{|P|}, B_1, b_1, B_2, b_2, \cdots, B_{|L|}, b_{|L|}, c$, such that there exists a parameterization of standard Cartesian coordinate system, such that points in $P$ have coordinates $(A_i, a_i)$, and lines in $L$ have equations $B_ix - b_iy = c$.
\end{definition}

\begin{remark}
 \emph{Naturality} in definition implies that parameters of points and lines are natural numbers, and \emph{normality} means all lines have the same free term $c$.
\end{remark}

\begin{definition}
We call normed natural configuration of points $P$ with coordinates $(A_i, a_i)$ and lines $L$ with equations $B_ix - b_iy = c$ \emph{prime}, if following conditions hold:
$$(A_1, c) = (A_2, c) = \ldots = (B_1, c) = (B_2, c) = \ldots = 1,
$$ 
$$(a_1, c) = (a_2, c) = \ldots = (b_1, c) = (b_2, c) = \ldots = 1.
$$ 
\end{definition}

\begin{lemma}[\bfseries on prime configuration]
The bipartite graph corresponding to prime configuration does not contain cycles.
\end{lemma}

\begin{proof}
Surely, point-line configuration does not contain cycles of length 4.
Suppose this graph contains cycle of length $2k, k \geqslant 3$ (because of biparticy, cycles of odd length are not possible). This cycle can be represented as 
$$
(A_1, a_1) \rightarrow 
(B_1, b_1) \rightarrow 
(A_2, a_2) \rightarrow 
\ldots \rightarrow 
(B_{k}, b_{k}) \rightarrow 
(A_1, a_1).
 $$
 We consider the sequence of $4$ edges following each other:
 $$
 (X, x) \rightarrow
 (P, p) \rightarrow
 (Y, y) \rightarrow
 (Q, q) \rightarrow
 (Z, z).
 $$
 One can deduce following equations then:
 $$
 XP - xp = c,
 $$
 $$
 PY - py = c,
 $$
 $$
 YQ - yq = c,
 $$
 $$
 QZ - qz = c.
 $$

It follows easily that
 $$
 py + c = PY \mid XYPQ = (xp + c)(yq + c),
 $$
so
 $$
 py + c \mid (xp + c)(yq + c) - (py + c)(qx + c) = c(x - y)(p - q).
 $$
Because the graph is prime we have $(py + c, c) = (py, c) = 1$, so the $c$ might be reduced, that's why
 $$
 py + c \mid |x - y| |p - q|.
 $$
 Similarly, we have
 $$
 yq + c \mid |y - z| |p - q|.
 $$
 At this moment one can declare, that we have chosen such part of the cycle, where the value of $y$ is maximal among $x, y, z$. One can always declare that because all the numbers are different --- if, for example, $x = y$, then $XP = xp + c = yp + c = YP$, so $X = Y$, i.e. points $(X, x), (Y, y)$ of the same part do coincide.
 
 So, let $y$ be the maximal number. Without loss of generality, let $p > q$ (case $q < p$ is equivalent). 
 Then we have $0 < (y - x)(p - q) < yp < yp + c$, but $(y - x)(p - q)$ is divisible on $yp + c$. Contradiction.
\end{proof}

\begin{remark}
Due to lack of cycles, graph corresponding to prime configuration is a forest, that is why it has number of edges less than number of vertices $|P| + |L| = O(\max(|P|, |L|))$.
\end{remark}

\begin{lemma}[\bfseries on a normed natural configuration]
Let configuration of points $P$ ($|P| = m$) and lines $L$ ($|L|$ = n) be a normed natural configuration. Then number of incidences, i.e set of edges $E$ in corresponding bipartite graph satisfies inequality 
$$
|E| \leqslant (m+n)\tau(c)^3,
$$
 where $\tau(c)$ is the number of divisors of $c$.
\end{lemma}
\begin{proof}
We notice first that $AB - ab = c$ implies $(AB, c) = (ab, c)$. 
Let us consider all possible quintuples $(l, l_1, l_2, l_3, l_4)$ of divisors of $c$ such that $l = l_1l_2 = l_3l_4$ holds. Let us consider classes
 $$
 P_{l_1, l_3} = 
 \{ 
 (A, a) \in P : 
 l_1 \mid A,\ 
 l_3 \mid a\}, 
 L_{l_2, l_4} = 
 \{ 
 (B, b) \in L : 
 l_2 \mid B,\ 
 l_4 \mid b\}
 $$
 As before, we encoded line $B_ix - b_iy = c$ as $(B_i, b_i)$.
 
 Clearly, $P$ and $L$ become union of such sets, not necessarily strict.

 Let us consider classes 
 $$
 E_{l, l_1, l_2, l_3, l_4} = 
 \{ 
 (A, a) \in P, (B, b) \in L: 
 (AB, c) = (ab, c) = l,\ 
 l_1 \mid A,\ 
 l_2 \mid B,\ 
 l_3 \mid a,\ 
 l_4 \mid b \}
 $$
 Cleary, set of edges $E$ becomes a union of such sets --- not necessarily strict. It is also clear, that edges from $E_{l, l_1, l_2, l_3, l_4}$ connect vertices from $P_{l_1, l_3}$ and $L_{l_2, l_4}$. From $l = l_1l_2 = l_3l_4$ it follows that one can define class parameters through (say) triple $(l, l_2, l_4)$. This gives bound $\tau(c)^3 $ on number of classes.
 
 For any class $E_{l, l_1, l_2, l_3, l_4}$
 we introduce following sets: 
 $$P_{l, l_1, l_2, l_3, l_4} = 
 \left\{
 \left(
 \frac{A}{l_1}, \frac{a}{l_3}
 \right) : (A, a) \in P_{l_1, l_3}
 \right\}, 
 L_{l, l_1, l_2, l_3, l_4} = 
 \left\{ 
 \left(
 \frac{
 B}{l_2}, 
 \frac{b}{l_4}
 \right) : (B, b) \in L_{l_2, l_4}
 \right\},
 $$ 
Clearly, for any fixed set $E_{l, l_1, l_2, l_3, l_4}$ equality $AB - ab = c$ implies
 $$
\frac{A}{l_1}\frac{B}{l_2} - 
\frac{a}{l_3}\frac{b}{l_4} = 
\frac{c}{l} = c'.
 $$

 We notice that configuration $P_{l, l_1, l_2, l_3, l_4}, L_{l, l_1, l_2, l_3, l_4}$ is prime by its construction, since we got rid off all common divisors of $c'$ with remaining numbers. Number of vertices in this configuration $|P_{l, l_1, l_2, l_3, l_4}| + |L_{l, l_1, l_2, l_3, l_4}|$ does not exceed $m + n$, and amount of edges is $E_{l, l_1, l_2, l_3, l_4}$. Since graph lacks cycles, amoung of edges does not exceed amount of vertices, and therefore $|E_{l, l_1, l_2, l_3, l_4}| \leqslant m + n$. Since all edges from $E$ partitioned into classes of sizes at most $m+n$, and number of classes is bounded by $\tau(c)^3 $, the bound $|E| \leqslant (m+n)\tau(c)^3 $ is proved.
\end{proof} 

\section{On Differences of Multiplicative Functions}
In this paragraph we prove the lemma on the differences between multiplicative functions, which is used later to estimate amount of the solutions to the equation $n - \varphi(n) = c$.

We first prove a helpful lemma.

\begin{lemma}[\bfseries on the number partitioning]
Let $n = x_1x_2\ldots x_k$ be such a number that any $x_i$ does not exceed $t \in \mathbb{R}_{+}$. Then such a partition $n = ab$ ($a = x_{i_1}x_{i_2}\ldots, b = x_{j_1}x_{j_2}\ldots $) exists such that $1 \leqslant a, b \leqslant \sqrt{n}\sqrt{t}$.
\end{lemma}
\begin{proof}
 From $n = x_1 \ldots x_k$ we have:
 $$
 \log{n} = \log{x_1} + \ldots + \log{x_k}.
 $$
 We split $x_i$'s into two groups with sums $s$ and $r$ such that the absolute value of their difference is minimal possible. Now we demonstrate that $|s - r| \leqslant \log{t}$ --- then $a = e^{s}, b = e^{r}$ would be the partition we desired to obtain. Suppose it does not hold. Without loss of generality $s \geqslant r$. We can try to move some $\log{x} \leqslant \log{t}$ from bigger sum to smaller sum: $\hat{s} = s - \log{x}, \hat{r} = r + \log{x}$. It is clear that $\hat{s} \leqslant \hat{r}$ (otherwise we would just make them closer and therefore their absolute difference was not minimal possible). Because the absolute difference was smallest possible, we have inequality: $\hat{r} - \hat{s} \geqslant s - r$, thus $2\log{x} \geqslant 2(s - r)$. But we supposed $s - r$ strictly exceeds $\log{t}$. Contradiction.
\end{proof}

Now the main result of this chapter:
 
 \begin{lemma}[\bfseries on the differences between multiplicative functions]
 Let $f:\mathbb{N} \rightarrow \mathbb{N}, g:\mathbb{N} \rightarrow \mathbb{N}$ be the two multiplicative functions, $c$ is a natural number, $t > 0$ is a real number and $N \subset \mathbb{N}$ is the set of natural numbers. Assume next conditions hold: 

 $(i)$ 
 $f(n) > g(n)$ for all $n > 1$.

 $(ii)$ $(f(n), g(n)) = (f(m), g(m)) \Leftrightarrow n = m$.

 $(iii)$ $f(n) - g(n) = c$ for all $n$ in $N$. 

 $(iv)$ For arbitrary $x > 0$ there exists no more than $O(x)$ such $n$ that $f(n) \leqslant x$.

 $(v)$ For any $n = \prod_{i}p_{i}^{\alpha_i} \in N$ it is true that $f(p_{i}^{\alpha_i})$ does not exceed $t$.

 Then it is true that $|N| = O (t\sqrt{c}\tau(c)^3 ) = O(tc^{\frac{1}{2} + o(1)})$.
 \end{lemma}

\begin{example}
One may check easily that pair of functions $f(n) = n, g(n) = \varphi(n)$ satisfies those conditions (namely, $(i), (ii)$ and $(iv)$).
\end{example}

 \begin{proof}
 Let elements of $N$ be $\{ n_1, n_2, \ldots\}$. It is clear that for any $n$ in $N$ inequality $f(n) \leqslant ct$ holds. Indeed, let us consider some primal divisor of $n$, say $p^{\alpha}$. Then $n = mp^{\alpha}$, where $(m, p^{\alpha}) = 1$. From $f(n) - g(n) = c$ it follows that $f(m)f(p^{\alpha}) - g(m)g(p^{\alpha}) = c$, i.e.
 $$
 c = 
 f(m)f(p^{\alpha}) - g(m)g(p^{\alpha}) = 
 \big(f(m) - g(m)\big)g(p^{\alpha}) + \big(f(p^{\alpha}) - g(p^{\alpha})\big)f(m) \geqslant $$
 $$
 \geqslant \big(f(p^{\alpha}) - g(p^{\alpha})\big)f(m) \geqslant
 f(m).
 $$ 
 Then $f(n) = f(mp^{\alpha}) = f(m)f(p^{\alpha}) \leqslant ct$.

 For any $n_i$ we build such partitioning $f(n_i) = f(a_i)f(b_i)$ (if $n = p_1^{\alpha_1}p_2^{\alpha_2}\ldots p_k^{\alpha_k}$, then we use number partitioning lemma for $f(n) = f(p_1^{\alpha_1})f(p_2^{\alpha_2})\ldots f(p_k^{\alpha_k})\ $) such that $$
 1 \leqslant f(a_i), f(b_i) \leqslant \sqrt{f(n_i)}\sqrt{t} \leqslant t\sqrt{c}.
 $$ 
 
 Now let $P = \{ (f(a), g(a))\}$ be set of points, and $L = \{ f(b)x - g(b)y = c)\}$ be set of lines. Clearly, because of conditions $(ii)$ and $(iv)$, and the fact $t\sqrt{c} \geqslant f(a) \geqslant g(a)$ (same for $f(b), g(b)$) both sets have size $O(t\sqrt{c})$.
 
 Summing amount of edges by classes, we obtain required bound on size of $N$.
 
 \end{proof}
 
 \begin{remark}
 One can replace condition $(v)$ with $(v')$: for any $n = \prod_{i}p_{i}^{\alpha_i} \in N$ inequality $g(p_{i}^{\alpha_i}) \leqslant t$ must hold. 
 Condition $(iii)$ might be replaced with $(iii)'$: for arbitrary $x$ there is only $O(x)$ such $n$ such that $g(n) \leqslant x$.
 
 With such conditions, lemma becomes applicable, for instance, to pair of functions $f(n)=\tau(n), g(n)=n$.
 \end{remark}

 \begin{remark}
 One can formulate similar result about sums of multiplicative functions $f(n) + g(n)$, the proof of which would require a corresponding change to lemma on prime configuration and lemma on normed natural configuration.
 \end{remark}
 
\section{On the number of solutions to the \texorpdfstring{$n - \varphi(n) = c$}{final} }
In this section we apply obtained results to find out the amount of solutions of $n - \varphi(n) = c$. Starting from here we consider $c$ to be fixed. Also we remind, that $G(k)$ is the number of ways to represent $k$ as a sum of two numbers.

\begin{theorem}
For given $c > 1$ the equation 
 $$
 n - \varphi(n) = c\ \ (*)
 $$
 has $G(c + 1) + O(c^{0.75 + o(1)})$ solutions.
\end{theorem}

\begin{proof}

We say that the number is primal, if it can be represented as $p^{a}$, where $p$ is the prime number. Clearly any natural $n > 1$ is a product of some primal numbers. We would first consider cases where $n$ is a product of no more than 2 primal numbers.\ 

\begin{lemma}[\bfseries $n$ --- primal]
Primal $n$ result in just $O(1)$ solutions.
\end{lemma}

\begin{proof}
Let $n = p^{a}$, then since $c > 1$, we have $a > 1$, thus $p^{a-1}(p-1) = c$ and $p|c$, and it is clear that $p$ is a greatest prime divisor of $c$, so the $a$ is defined explicitly. It gives us $O(1)$ solutions.
\end{proof}

\begin{lemma}[\bfseries $n$ --- product of $2$ primal numbers]

Those $n$ which are products of 2 primal numbers result in $G(c + 1) + O(\ln^{2}c)$ solutions.
\end{lemma}

\begin{proof}
 
Consider $n$ of the form $n = p^{a}q^{b}$.
If $a = b = 1$ holds, then $pq - \varphi(pq) = p + q - 1$, so $c + 1 = p + q$, which results in $G(c + 1)$ solutions. If one of them (say $a$) is greater than $1$, then we have $p | c$, and iterating over $a$ (which is obviously less than $\omega(c)$) gives us equations $q^{b-1}(q + p - 1) = \frac{c}{p^{a-1}}$, each of which has just $O(1)$ solutions, when $b = 1$, and $O(\omega(c))$ solutions when $b > 1$ and $q | c$, so we have just $O(\omega^{2}(c))$ solutions in common.

As a result we have $G(c + 1) + O(\ln^{2}c)$ solutions.
\end{proof}

Because of these lemmas we can now consider cases when $n$ is a product of at least $3$ primal numbers. Our goal now is to show, that there exists no more than $O(c^{0.75 + o(1)})$ solutions of equation $(*)$ in this case.

Let $N(c)$ be the number of such solutions $n$, which are products of at least 3 primal numbers. Let $N^{*}(c)$ be the number of such solutions $n$ from $N(c)$ which are square-free. Then the following holds:
$$
N(c) \leqslant \sum_{d|n}N^*(d){} = \sum_{d|n}N^*\left(\frac{c}{d}\right).
$$
We will demonstrate that $N^*(c) = O(c^{1 - \epsilon + o(1)})$, $\epsilon > 0$, and same holds for its sum over divisors of $c$, so we will only consider the case when $n$ is square-free in further text. Indeed, since $c^{1 - \epsilon}$ is monotonically increasing, one has
$$
\sum_{d|n} 
N^*\left(\frac{c}{d}\right) 
\leqslant 
\sum_{d=1}^{\tau}
\left( 
\frac{c}{d}
\right)^{1 - \epsilon + o(1)} = 
c^{1 - \epsilon + o(1)}
\sum_{d=1}^{\tau}
\frac{1}{d^{1 - \epsilon + o(1)}} 
\ll 
c^{1 - \epsilon + o(1)}
\frac{\tau^{\epsilon}}{\epsilon} =
c^{1 - \epsilon + o(1)},
$$
where $\tau = O(e^{\frac{\ln{c}}{\ln\ln{c}}})$ is a number of divisors of $c$. Since bounding of $N^*(c)$ leads to bounding of $N(c)$, we are only estimating $N^*(c)$ in further text.

Note that if $n = Ap$ then $n - \varphi(n) = Ap - \varphi(A)\varphi(p)$ holds, i.e.
$$Ap - \varphi(A)(p-1) = c.$$
The value of $Ap - \varphi(A)(p-1)$ might also be expressed as: 
$$ Ap - \varphi(A)(p-1) = 
(A-\varphi(A))p + \varphi(A) = 
\big(A-\varphi(A)\big)(p-1) + A,$$ from where it follows that $A < c$.

Now we represent $n$ as $n = Bpq$, where $p$ and $q$ are some prime divisors of $n$. Consider 2 cases:

\medskip

{\bf 1) \texorpdfstring{$n$ can be expressed as $n = Bpq$, where $B < c^{0.75 + o(1)}$ holds}{smooth}}

Note that if $B$ is some fixed number such that $1 < B < c^{0.75 + o(1)}$, then there exists no more than $c^{o(1)}$ such $n$, such that $n$ can be expressed as $n = Bpq$ ($p, q$ --- primes) and satisfying equation $(*)$.

Indeed, let $B > 1$ be fixed. Then $(*)$ with such $n$ is equivalent to:
 $$
 Bpq - \varphi(B)(p - 1)(q - 1) = c,
 $$
 which, on it's turn, is equivalent to
 $$
 \Big(
 \big(B - \varphi(B)\big)p + \varphi(B)
 \Big)
 \Big(
 \big(B - \varphi(B)\big)q + \varphi(B)
 \Big)
 =
 \big(B - \varphi(B)\big)c + B\varphi(B).
 $$
 Since $B$ is bounded, RHS does not exceed $c^{2}$. So it has no more than $e^{O(\frac{\ln{c}}{\ln\ln{c}})} = c^{o(1)}$ solutions. Summing up by $B$, we get estimate $O(c^{0.75 + o(1)})$ on amount of solutions.
 
 \medskip

{\bf 2) \texorpdfstring{$n$ cannot be expressed as $n = Bpq$, where $B < c^{0.75 + o(1)}$ holds}{nonsmooth}}

In this case we can assume that if $n = Bpq$ ($p$, $q$ --- primes), then $B > c^{0.75 + o(1)}$. 

From $n = Bpq = (Bp)q$ it follows that $Bp \leqslant c$, and $p \leqslant 
\frac{c}{B} \leqslant c^{\frac{1}{4} + o(1)}$. Since we could have taken $p$ as any prime divisor of $n$, this inequality holds for all prime divisors of $n$.

 We now apply lemma on differences between multiplicative functions to functions $f(n) = n, g(n) = \varphi(n)$ and $t = c^{\frac{1}{4} + o(1)}$. It gives us bound $c^{0.75 + o(1)}$ on number of solutions immediately

Summing up bounds from both cases we have bound on $N^*(c)$ equal to
$
O(c^{0.75 + o(1)}),
$
which we desired to prove.

\end{proof}

\begin{remark}
One may check, that the "obstacle" because of which we can not have better bound bound than $c^{0.75 + o(1)}$, is a case where $n$ is a product of $5$ prime divisors, each of which is roughly $c^{\frac{1}{4}}$. As soon as bound in this case is improved, the overall bound will be improved immediately.
\end{remark}

\begin{remark}
As we noticed, it is enough to consider only square-free $n$.
Let $M_k$ be a set of such $n$, which may be represented as $n = p_1p_2\ldots p_k$ such that $(*)$ holds. Then we have a hypothesis that for $k \geqslant 2$ holds estimation $|M_k| = O(c^{\frac{1}{k - 1} + o(1)})$. 
One may check (by Pigeonhole principle and lower bounds on amount of primes in the interval) that those bounds are true in "average" case (meaning those bounds are lower-bounds in those cases).

This hypothesis is easily verifiable in cases $k = 2$ and $k = 3$: when $k = 2$ it straitghly follows from equality $p + q = c + 1$.

For other $k$ one can use representation of $n$ as $n = Bpq$ (as we did above) and thus getting unconditional inequalities in the form
$$
|M_k| < c^{\frac{k - 2}{k - 1} + o(1)}.
$$
For $k = 3$ this bound coincides with the one stated by the hypothesis. 

\medskip

Using lemma on the differences between multiplicative functions one may obtain inequalities of such type:
$|M_{k}'| < c^{\frac{1}{2} + \epsilon_{k}}$, where $M_k'$ are such members of $M_k$, which do not have "too big"\ prime divisors, and $\epsilon_k$ depends on what we mean by "too big" prime divisor. In other words, the main complexity of the problem is about non-smooth numbers $n$.
\end{remark}

\begin{remark}
If one considers only even $c$, then we have a restriction that $n$ from $(*)$ is also even, which leads to inequality $c < n < 2c$. Obviously the "main term" of the sum, i.e. $G(c + 1)$ would be $O(1)$, because $c + 1$ is odd. By analogous implications one may obtain bound $O(c^{0.6 + o(1)})$ on a number of solutions. Despite that, one has a hypothesis that there exists just $c^{o(1)}$ solutions of $(*)$ when $n$ is even.
\end{remark}

\printbibliography[
heading=bibintoc,
] 

\noindent{A.S.~Semchankau

The Steklov Institute of Mathematics

119991, Russian Federation, Moscow, Ulitsa Gubkina, 8}

{\tt aliaksei.semchankau@gmail.com}

\end{document}